\documentclass[english,15pt]{article}
\usepackage[T1]{fontenc}
\usepackage[latin1]{inputenc}
\usepackage{geometry}
\usepackage{float}
\usepackage{mathrsfs}
\usepackage{amsmath}
\usepackage{amssymb}
\usepackage{graphicx}
\usepackage{hyperref}
\usepackage[color=yellow]{todonotes}
\hypersetup{
    colorlinks=true,
    linkcolor=blue,
    filecolor=magenta,      
    urlcolor=cyan,
    citecolor = blue,
}

\usepackage{yhmath}
\makeatletter

\usepackage{bm}
\numberwithin{equation}{section}

\newcommand{\seminorm}[1]{\left\lvert\hspace{-1 pt}\left\lvert\hspace{-1 pt}\left\lvert {#1}\right\lvert\hspace{-1 pt}\right\lvert\hspace{-1 pt}\right\lvert}

\floatstyle{ruled}
\newfloat{algorithm}{tbp}{loa}
\providecommand{\algorithmname}{Algorithm}
\floatname{algorithm}{\protect\algorithmname}

\usepackage{algorithm}
 \usepackage{algorithmicx}

\usepackage{amsthm}

\usepackage{mathrsfs}
\usepackage{amssymb}

\usepackage{amsfonts}

\usepackage{epsfig}

\usepackage{bm}

\usepackage{mathrsfs}

\usepackage{enumerate}

\@ifundefined{definecolor}{\@ifundefined{definecolor}
 {\@ifundefined{definecolor}
 {\usepackage{color}}{}
}{}
}{}

\usepackage{subfig}\usepackage[all]{xy}

\newtheorem{theorem}{Theorem}[section]
\newtheorem{lem}{Lemma}[section]

\newtheorem{cor}{Corollary}[section]
\newcounter{hypA}
\newenvironment{hypA}{\refstepcounter{hypA}\begin{itemize}
  \item[({\bf A\arabic{hypA}})]}{\end{itemize}}
\newcounter{hypB}

\newcounter{hypD}

\usepackage{babel}\date{}

\newcommand\smallO{
  \mathchoice
    {{\scriptstyle\mathcal{O}}}
    {{\scriptstyle\mathcal{O}}}
    {{\scriptscriptstyle\mathcal{O}}}
    {\scalebox{.7}{$\scriptscriptstyle\mathcal{O}$}}
  }

\makeatother

\begin{document}

\begin{center}

{\Large \textbf{Asymptotic Variance in the Central Limit Theorem for Multilevel Markovian Stochastic Approximation}}

\vspace{0.5cm}

AJAY JASRA \& ABYLAY ZHUMEKENOV

{\footnotesize 
School of Data Science, The Chinese University of Hong Kong, Shenzhen, Shenzhen, CN.}
{\footnotesize E-Mail:\,} \texttt{\emph{\footnotesize ajayjasra@cuhk.edu.cn,  abylayzhumekenov@cuhk.edu.cn}}

\begin{abstract}
In this note we consider the finite-dimensional parameter estimation problem associated to inverse problems.
In such scenarios,  one seeks to maximize the marginal likelihood associated to a Bayesian model.
This latter model is connected to the solution of partial or ordinary differential equation.  As such,  there are two
primary difficulties in maximizing the marginal likelihood (i) that the solution of differential equation is not always analytically tractable and (ii) neither is the marginal likelihood.   Typically (i) is dealt with using a numerical solution of the differential equation,  leading to a numerical bias and (ii) has been well studied in the literature using,  for instance,  Markovian stochastic approximation (e.g.~\cite{andr}).  It is well-known,  e.g.~\cite{ub_grad_new,ub_inv,disc_models},  that to reduce the computational
effort to obtain the maximal value of the parameter,  one can use a hierarchy of solutions of the differential equation and combine with stochastic gradient methods.  The approaches of \cite{ub_grad_new,ub_inv} do exactly this.  In this paper we consider the asymptotic variance in the central limit theorem,  associated to estimates from 
\cite{ub_inv} and find bounds on the asymptotic variance in terms of the precision of the solution of the differential equation.  The significance of these bounds are the that they provide missing theoretical guidelines on how to set simulation parameters in \cite{ub_grad_new,ub_inv}; that is,  these appear to be the first mathematical results which help to run the methods efficiently in practice.
\end{abstract}

\end{center}

\section{Introduction}

We consider the problem of finding the root of the function
$$
h(\theta) = \int_{\mathsf{X}}H(\theta,x)\pi_{\theta}(x)dx
$$
where for each $\theta\in\Theta\subseteq\mathbb{R}$,  $\pi_{\theta}$ is a probability density on the space $(\mathsf{X},\mathscr{X})$,  $dx$ is a $\sigma-$finite measure and $H:\Theta\times\mathsf{X}\rightarrow\Theta$ is, for each $\theta\in\Theta$,   a $\pi_{\theta}-$integrable function.   We assume that there exists a unique $\theta^{\star}\in\Theta$ which such that $h(\theta^{\star})=0$.   In the kind of problems that we are interested in,
$H(\theta,x)$ is the gradient w.r.t.~$\theta$ of the log of an un-normalized probability density; examples can be found in \cite{disc_models}.  The problem given here can be solved using Markovian stochastic approximation (MSA) (e.g.~\cite{andr}) especially if $\pi_{\theta}$ cannot be sampled from exactly.  We review the MSA method in the next Section.
In many applications $\pi_{\theta}$ and $H(\theta,x)$ is associated with the solution of a partial or ordinary differential equation and in such scenarios one replaces the exact solution with a provably convergent numerical solution.  Then one can proceed by trying to solve the approximated problem using MSA;  this is detailed in the next Section.

It is well-known (see for instance \cite{ml_rev}) that one can reduce the computational effort to compute the root (maximizer) associated to the approximated $h(\theta)$.  Instead of considering simply one MSA algorithm with a numerical solution of the differential equation of a given accuracy,  one considers several solutions simultaneously. This leads to multilevel Markovian stochastic approximation algorithms as was developed originally in \cite{ub_grad_new} and applied directly to partial/ordinary differential equation problems in \cite{ub_inv}.
The articles \cite{ub_grad_new,ub_inv} focus on developing new methodology and proving that unbiased estimates of $\theta^{\star}$ can be obtained,  but critically,  rely on mathematical conjectures in order to set certain simulation parameters.  The optimal setting of said simulation parameters is often key to ensuring a reduction in computational effort,  of using multilevel MSA versus MSA on its own.  In this note we intend to fill in the theoretical gap from \cite{ub_inv}.

In this note we consider the asymptotic variance expression from the central limit theorem (CLT) proved by \cite{fort} for MSA.  The expression is adapted for the estimators that are under study in this work,  i.e.~of multilevel type (explained in the next Section).
In particular and under mathematical assumptions,  we provide a non-trivial upper-bound for the asymptotic variance in the CLT,  in terms of the accuracy of the numerical solver of the differential equation.  This upper-bound is critical in providing a mathematical justification of how to set the simulation parameters in \cite{ub_inv}.
It should be noted that our approach is limited to $\theta\in\Theta\subseteq\mathbb{R}$. The reason for this is because the general ($d-$dimensional) case in \cite{fort} has the asymptotic (co)variance in the form of a solution to Lyapunov's equation.  Whilst under fairly general assumptions one has a unique to solution to Lyapunov's equation,   the general solution is analytically intractable,  at least in a form that could be used to bound it in an appropriate norm.  Thus we consider only the one-dimensional case.   We also remark that we rely on asymptotics such as the CLT,  instead of finite time (algorithm time) bounds as the latter do not appear to be in the literature and is more general problem than the one that we seek to address in this note.

This note is structured as follows.  In Section \ref{sec:setting} we give details of the problem of interest and the algorithm which we study.  Section \ref{sec:math} gives the mathematical set-up for our main result,  including notations and assumptions.  In Section \ref{sec:main} our main result is given along with some implications.
Appendix \ref{app:tech} features several technical results used the prove the theorem in Section \ref{sec:main}.

\section{Setting}\label{sec:setting}

\subsection{Approach}

 In many applications in practice,  obtaining $\theta^{\star}$ can be solved by using MSA as in e.g.~\cite{andr}.
For any $\theta\in\Theta$ let $K_{\theta}:\mathsf{X}\times\mathscr{X}\rightarrow[0,1]$ be a $\pi_{\theta}-$invariant Markov kernel and let $(\gamma_n)_{n\in\mathbb{N}}$ be a sequence of non-negative numbers such
that $\sum_{n\geq 1}\gamma_n=\infty$,  $\sum_{n\geq 1}\gamma_n^2<\infty$, $\log(\gamma_n/\gamma_{n-1})=\smallO(\gamma_n)$; this latter assumption is \cite[Condition  4a]{fort},  which we use later on.
Then given $(\theta_0,x_0)\in\Theta\times\mathsf{X}$ one can
run the following procedure,  for $n\geq 1$:
\begin{enumerate}
\item{Simulate $X_n|(\theta_0,x_0),\dots,(\theta_{n-1},x_{n-1})$ using $K_{\theta_{n-1}}(x_{n-1},\cdot)$.}
\item{Update 
$$
\theta_n = \theta_{n-1} + \gamma_n H(\theta_{n-1},X_n)
$$
and set $n=n+1$ and go to the start of 1..}
\end{enumerate}
In many cases one employs the idea of reprojection (e.g.~\cite{andr,fort}) and we also do this.   Let $(\mathscr{K}_n)_{n\geq 0}$ be a sequence of compact subsets of $\Theta$ such that for each $n\geq 1$,   $\mathscr{K}_{n-1}\subseteq\mathscr{K}_n$ and
$$
\bigcup_{n\geq 0}\mathscr{K}_n = \Theta.
$$
Then we modify the above procedure as follows: set $(\theta_0,x_0,\psi_0)\in\mathscr{K}_0\times\mathsf{X}\times\{0\}$ then we consider the procedure or $n\geq 1$:
\begin{enumerate}
\item{Simulate $X_n|(\theta_0,x_0,\psi_0),\dots,(\theta_{n-1},x_{n-1},\psi_{n-1})$ using $K_{\theta_{n-1}}(x_{n-1},\cdot)$.}
\item{Set
$$
\theta_{n-1/2} = \theta_{n-1} + \gamma_n H(\theta_{n-1},X_n).
$$
}
\item{Update:
$$
(\theta_n,\psi_n) = 
\left\{\begin{array}{ll}
(\theta_{n-1/2},\psi_{n-1}) &  \text{if}~\theta_{n-1/2}\in\mathscr{K}_{\psi_{n-1}},\\
(\theta_0,\psi_{n-1}+1) &\text{otherwise} 
\end{array} \right.
$$
and set $n=n+1$ and go to the start of 1..}
\end{enumerate}

\subsection{Approximation}

In several applications,  such as \cite{ub_grad_new,ub_inv} working with the exact model is not possible.
Instead,  one has a scalar parameter $l\in\mathbb{N}_0$,  a sequence of functions $H_l(\theta,x)$ and sequence
of probability densities $\pi_{\theta,l}$ on $(\mathsf{X},\mathscr{X})$ such that for each $\theta\in\Theta$:
$$
h_l(\theta) := \int_{\mathsf{X}}H_l(\theta,x)\pi_{\theta,l}(x)dx~\text{and}~\lim_{l\rightarrow\infty}h_l(\theta) = h(\theta).
$$
We shall assume that there exists a unique $\theta^{\star}_l\in\Theta$ such that $h_l(\theta^{\star}_l)=0$
and note that,  we have $\lim_{l\rightarrow\infty}\theta^{\star}_l=\theta^{\star}$.
It should be understood that
\begin{itemize}
\item{$h_{l}(\theta)$ is closer,  in terms of $L_1-$distance than $h_{l-1}(\theta)$ to $h(\theta)$.}
\item{The computational cost of working with $(H_l,\pi_{\theta,l})$ is larger than working with $(H_{l-1},\pi_{\theta,l-1})$.}
\end{itemize}
Explicit examples can be found in \cite{ub_grad_new,ub_inv} and we note in \cite{ub_grad_new}, the space $(\mathsf{X},\mathscr{X})$ would also change with $l$.

For any $(\theta,l)\in\Theta\times\mathbb{N}_0$ let $K_{\theta,l}:\mathsf{X}\times\mathscr{X}\rightarrow[0,1]$ be a $\pi_{\theta,l}-$invariant Markov kernel.  We can use the following:  set $(\theta_{0,l},x_0,\psi_0)\in\mathscr{K}_0\times\mathsf{X}\times\{0\}$ then we consider the procedure for $n\geq 1$:
\begin{enumerate}
\item{Simulate $X_n|(\theta_{0,l},x_0,\psi_0),\dots,(\theta_{n-1,l},x_{n-1},\psi_{n-1})$ using $K_{\theta_{n-1,l},l}(x_{n-1},\cdot)$.}
\item{Set
$$
\theta_{n-1/2,l} = \theta_{n-1,l} + \gamma_n H(\theta_{n-1,l},X_n).
$$
}
\item{Update:
$$
(\theta_{n,l},\psi_n) = 
\left\{\begin{array}{ll}
(\theta_{n-1/2,l},\psi_{n-1}) &  \text{if}~\theta_{n-1/2,l}\in\mathscr{K}_{\psi_{n-1}},\\
(\theta_{0,l},\psi_{n-1}+1) &\text{otherwise} 
\end{array} \right.
$$
and set $n=n+1$ and go to the start of 1..}
\end{enumerate}

In the case of using unbiased and multilevel estimation schemes,  one is often interested in approximating the collapsing sum, for $L$ fixed or growing:
$$
\theta_L^{\star} = \theta_0^{\star} + \sum_{l=1}^L \left\{\theta_l^{\star}-\theta_{l-1}^{\star}\right\}.
$$
The idea is that for $l=0$ one can use the method detailed above and for each summand on the R.H.S.~one can use a coupled version,  which we will now introduce,   independently across the summands. 
For any $(\theta,\overline{\theta},l,\overline{l})\in\Theta^2\times\mathbb{N}_0^2$,  $\boldsymbol{\theta}=(\theta,\overline{\theta})$,  $\mathbf{l}=(l,\overline{l})$, let $\check{K}_{\boldsymbol{\theta},\mathbf{l}}:\mathsf{X}^2\times\mathscr{X}^2\rightarrow[0,1]$ be a Markov kernel such that
for any $(x,\overline{x},\mathsf{A})\in\mathsf{X}^2\times\mathscr{X}$ we have that
\begin{eqnarray}
\int_{\mathsf{A}\times\mathsf{X}} \check{K}_{\boldsymbol{\theta},\mathbf{l}}\left((x,\overline{x}),d(x,\overline{x}')\right) & = &  \int_{\mathsf{A}} K_{\theta,l}(x,dx')~\text{and}\label{eq:coup1}\\
\int_{\mathsf{X}\times\mathsf{A}} \check{K}_{\boldsymbol{\theta},\mathbf{l}}\left((x,\overline{x}),d(x,,\overline{x}')\right) & = &  \int_{\mathsf{A}} K_{\overline{\theta},\overline{l}}(\overline{x},d\overline{x}').\label{eq:coup2}
\end{eqnarray}
Note that for $l\in\mathbb{N}_0$ given,  if $\overline{l}=l-1$ we use the subscript $l$ in 
$\check{K}_{\boldsymbol{\theta},l}$ instead of writing $\check{K}_{\boldsymbol{\theta},\mathbf{l}}$.  An obvious example of
such a kernel is 
$$
\check{K}_{\boldsymbol{\theta},\mathbf{l}}\left((x,\overline{x}),d(x,\overline{x}')\right)= 
K_{\theta,l}(x,dx')K_{\overline{\theta},\overline{l}}(\overline{x},d\overline{x}')
$$
however more sophisticated examples are possible;  see \cite{disc_models}
for instance.
We can use the following to estimate $\theta_l^{\star}-\theta_{l-1}^{\star}$.  Set $(\theta_{0,l},\overline{\theta}_{0,l-1},x_{0,l},\overline{x}_{0,l-1},\psi_{0,l})\in\mathscr{K}_0^2\times\mathsf{X}^2\times\{0\}$ and
we shall denote in 1.~below $|\cdots$ to mean conditioning on 
$$
(\theta_{0,l},\overline{\theta}_{0,l-1},x_{0,l},\overline{x}_{0,l-1},\psi_{0,l}),\dots
(\theta_{n-1,l},\overline{\theta}_{n-1,l-1},x_{n-1,l},\overline{x}_{n-1,l-1},\psi_{n-1,l}).
$$
We consider the procedure for $n\geq 1$:
\begin{enumerate}
\item{Simulate $(X_{n,l},\overline{X}_{n,l-1})|\cdots$ using $\check{K}_{\boldsymbol{\theta}_{n-1},l}\left((x_{n,l},\overline{x}_{n,l-1}),\cdot\right)$.}
\item{Set
\begin{eqnarray*}
\theta_{n-1/2,l} & =&  \theta_{n-1,l} + \gamma_n H_l(\theta_{n-1,l},X_{n,l}) \\
\overline{\theta}_{n-1/2,l-1} & =&  \overline{\theta}_{n-1,l-1} + \gamma_n H_{l-1}(\overline{\theta}_{n-1,l-1},\overline{X}_{n,l-1})
\end{eqnarray*}
}
\item{Update:
$$
(\theta_{n,l},\overline{\theta}_{n,l-1},\psi_{n,l}) = 
\left\{\begin{array}{ll}
(\theta_{n-1/2,l},\overline{\theta}_{n-1/2,l-1},\psi_{n-1,l}) &  \text{if}~(\theta_{n-1/2,l},\overline{\theta}_{n-1/2,l-1})\in\mathscr{K}_{\psi_{n-1,l}}^2,\\
(\theta_{0,l},\overline{\theta}_{0,l-1},\psi_{n-1,l}+1) &\text{otherwise} 
\end{array} \right.
$$
and set $n=n+1$ and go to the start of 1..}
\end{enumerate}

\section{Mathematical Set-Up}\label{sec:math}

\subsection{Notation}

In order to continue with our analysis,  we need several notations.  On an arbitrary measurable space
$(\mathsf{E},\mathscr{E})$ for any finite measure $\mu$ on $(\mathsf{E},\mathscr{E})$ and
any $\mu-$integrable $f:\mathsf{E}\rightarrow\mathbb{R}^k$ we write $\mu(f)=\int_{\mathsf{E}}f(x)\mu(dx)$.
Note that in the context of MSA we will write (e.g.) $\pi_{\theta}(f)=\int_{\mathsf{X}}f(x)\pi_{\theta}(x)dx$.
The collection of probability measures on $(\mathsf{E},\mathscr{E})$  are denoted $\mathscr{P}(\mathsf{E})$.
On a product space $(\mathsf{E}\times\mathsf{E},\mathscr{E}\otimes\mathscr{E})$ we will write for any 
finite measure $\mu$ on $(\mathsf{E}\times\mathsf{E},\mathscr{E}\otimes\mathscr{E})$ and any 
$\mu-$integrable $f,g:\mathsf{E}^2\rightarrow\mathbb{R}^k$:
$$
\mu\left(f\otimes g\right) = \int_{\mathsf{E}^2} f(x)g(y)\mu\left(d(x,y)\right).
$$
For a Markov kernel on $M$ on $(\mathsf{E},\mathscr{E})$ we will set $M(f)(x) = \int_{\mathsf{E}}f(y)M(x,dy)$,  assuming $M(|f|)(x)<+\infty$, $|\cdot|$ is the $L_1-$norm.
In addition,  for $\mu$, $M$ and $f$ as described we write $\mu M(f) = \mu(M(f))$ assuming that 
$M(f)$ is $\mu-$integrable.  Let $W:\mathsf{E}\rightarrow[1,
\infty)$ then for any $f:\mathsf{E}\rightarrow\mathbb{R}$ we set
$$
|f|_W = \sup_{x\in\mathsf{E}}\frac{|f(x)|}{W(x)}.
$$
If $W=1$ we write $|f|_{\infty}$ instead of $|f|_1$.
For $(\mu,\xi)\in\mathscr{P}(\mathsf{E})$ we set
$$
\|\mu-\xi\|_W = \sup_{|f|\leq W}|\mu(f)-\xi(f)|.
$$
For $M_1,M_2$ Markov kernels we write
$$
\seminorm{M_1-M_2}_{W} = \sup_{x\in\mathsf{E}}\frac{\|M_1(x,\cdot)-M_2(x,\cdot)\|_W}{W(x)}.
$$
For $f:\mathsf{E}\rightarrow\mathbb{R}$,  if $|f|_W<+\infty$,  we say that $f\in\mathscr{L}_W$ and if
$|f|_{\infty}<+\infty$,  we write $f\in\mathscr{B}_b(\mathsf{E})$.

In the context of the MSA procedures discussed,  we shall assume that for $\boldsymbol{\theta},\mathbf{l}$ 
given $\check{K}_{\boldsymbol{\theta},\mathbf{l}}$ admits an invariant measure
$\check{\pi}_{\boldsymbol{\theta},\mathbf{l}}$ and due to the properties in \eqref{eq:coup1}-\eqref{eq:coup2}
we have that for any $f:\mathsf{X}\rightarrow\mathbb{R}$ that is both $\pi_{\theta,l}-$ and $\pi_{\overline{\theta},\overline{l}}-$integrable
$$
\check{\pi}_{\boldsymbol{\theta},\mathbf{l}}\left(f\otimes 1\right) = \pi_{\theta,l}(f)~\text{and}~
\check{\pi}_{\boldsymbol{\theta},\mathbf{l}}\left(1\otimes f\right) = \pi_{\overline{\theta},\overline{l}}(f).
$$
If for $l$ given,  $\overline{l}=l-1$ we write $\check{\pi}_{\boldsymbol{\theta},l}$ instead of $\check{\pi}_{\boldsymbol{\theta},\mathbf{l}}$.
If we consider the integration of any $\theta-$dependent function $J:\Theta\times\mathsf{X}\rightarrow\mathbb{R}^k$ we will write 
$$
 \pi_{\theta,l}(J_{\theta}) = \int_{\mathsf{X}}J(\theta,x) \pi_{\theta,l}(x)dx.
$$
A similar convention is used when integrating w.r.t.~$\check{K}_{\boldsymbol{\theta},\mathbf{l}}$ or
$K_{\theta,l}$.  For a function $J(\theta,x)$  we denote by
$$
\widehat{J}_{l}(\theta,x) - K_{\theta,l}(\widehat{J}_{\theta,l})(x) = J(\theta,x) - \pi_{\theta,l}(J_{\theta})
$$
where $\widehat{J}_{l}(\theta,x)$ is the solution of the Poisson equation.

\subsection{Asymptotic Variance}

We consider the case that $\Theta=\mathbb{R}$;  we have explained why this is the case earlier in the article.  \cite{fort} proves under assumptions that
\begin{equation}\label{eq:clt}
\gamma_n^{-1/2}\left(\{\theta_{n,l}-\overline{\theta}_{n,l-1}\} - 
\{\theta^{\star}_l-\theta^{\star}_{l-1}\}
\right) \rightarrow_d \mathcal{N}\left(0,\Sigma_{\boldsymbol{\theta}_l^{\star},l}\right)
\end{equation}
where $\boldsymbol{\theta}_l^{\star}=(\theta^{\star}_l,\theta^{\star}_{l-1})$ and
\begin{eqnarray*}
\Sigma_{\boldsymbol{\theta}_l^{\star},l} & = & -\left(2\frac{dh_l(\theta^{\star}_l)}{d\theta}\right)^{-1}
\check{\pi}_{\boldsymbol{\theta}_l^{\star},l}\left(
\widehat{H}_{\theta_l^{\star},l}^2\otimes 1 - 
K_{\theta^{\star}_l,l}\left(\widehat{H}_{\theta_l^{\star},l}\right)^2\otimes 1
\right) - \\
& & \left(2\frac{dh_{l-1}(\theta^{\star}_{l-1})}{d\theta}\right)^{-1}
\check{\pi}_{\boldsymbol{\theta}_l^{\star},l}\left(
1\otimes\widehat{H}_{\theta_{l-1}^{\star},l-1}^2 - 
1\otimes K_{\theta^{\star}_{l-1},l-1}\left(\widehat{H}_{\theta_{l-1}^{\star},l-1}\right)^2
\right) + \\
& &
2\left(
\frac{dh_l(\theta^{\star}_l)}{d\theta} + 
\frac{dh_{l-1}(\theta^{\star}_{l-1})}{d\theta}\right)^{-1}
\check{\pi}_{\boldsymbol{\theta}_l^{\star},l}\left(
\widehat{H}_{\theta_l^{\star},l}\otimes \widehat{H}_{\theta_{l-1}^{\star},l-1} - 
K_{\theta^{\star}_l,l}\left(\widehat{H}_{\theta_l^{\star},l}\right)
\otimes K_{\theta^{\star}_{l-1},l-1}\left(\widehat{H}_{\theta_{l-1}^{\star},l-1}\right)\right).
\end{eqnarray*}
Note that within the assumptions of \cite{fort} is that 
$$
\max\left\{
\frac{dh_l(\theta^{\star}_l)}{d\theta}, 
\frac{dh_{l-1}(\theta^{\star}_{l-1})}{d\theta}
\right\} < 0 
$$
which we assume also.   We note that the convergence in distribution in \eqref{eq:clt} is in fact conditional on $(\theta_{n,l},\theta_{n,l-1})$ converging to $(\theta^{\star}_l,\theta^{\star}_{l-1})$.

\subsection{Assumptions}

Below we will use $C$ to denote a generic finite constant that is $(\theta,l,x)\in\Theta\times\mathbb{N}_0\times\mathsf{X}$ independent.  Also $\Delta_l=2^{-l}$,  which in practice relates to the accuracy of a numerical
solver of a differential equation (see e.g.~\cite{beskos1}).
Implicit in our assumptions is the differentiability of $H_l$, $\pi_{\theta,l}$ etc,  which are not stated.

\begin{hypA}\label{ass:1}
\begin{enumerate}
\item{For any $(\theta,l)\in\Theta\times\mathbb{N}_0$ let $\nu_{\theta,l}\in\mathscr{P}(\mathsf{X})$.
There exists $(\epsilon,\mathsf{C})\in(0,1)\times\mathscr{X}$ such that for any  
$(\theta,l,x)\in\Theta\times\mathbb{N}_0\times\mathsf{X}$
$$
K_{\theta,l}(x,dx') \geq \epsilon \nu_{\theta,l}(dx').
$$
Moreover we have $\inf_{(\theta,l)\in\Theta\times\mathbb{N}_0}\nu_{\theta,l}(\mathsf{C})=C>0$.
\label{ass:11}
}
\item{
For any $(\theta,l)\in\Theta\times\mathbb{N}_0$ let $V_{\theta,l}:\mathsf{X}\rightarrow[1,\infty)$.
There exists $(\lambda,b)\in(0,1)\times(0,\infty)$ such that for any $(\theta,l,x)\in\Theta\times\mathbb{N}_0\times\mathsf{X}$
$$
K_{\theta,l}(V_{\theta,l})(x) \leq \lambda V_{\theta,l}(x) + b\mathbb{I}_{\mathsf{C}}(x)
$$
where $\mathsf{C}$ is as in \ref{ass:11}.  \label{ass:12}
}
\item{For any $\theta\in\Theta$ let $\nu_{\theta}\in\mathscr{P}(\mathsf{X})$.
There exists $(\epsilon,\mathsf{C})\in(0,1)\times\mathscr{X}$ such that for any  
$(\theta,x)\in\Theta\times\mathsf{X}$
$$
K_{\theta}(x,dx') \geq \epsilon \nu_{\theta}(dx').
$$
Moreover we have $\inf_{\theta\in\Theta\times\mathbb{N}_0}\nu_{\theta}(\mathsf{C})=C>0$.
\label{ass:13}
}
\item{
For any $\theta\in\Theta$ let $V_{\theta}:\mathsf{X}\rightarrow[1,\infty)$.
There exists $(\lambda,b)\in(0,1)\times(0,\infty)$ such that for any $(\theta,x)\in\Theta\times\mathsf{X}$
$$
K_{\theta}(V_{\theta})(x) \leq \lambda V_{\theta}(x) + b\mathbb{I}_{\mathsf{C}}(x)
$$
where $\mathsf{C}$ is as in \ref{ass:13}. \label{ass:14}
}
\item{For $V_{\theta,l}$ as in \ref{ass:12} and $V_{\theta}$ as in  \ref{ass:14} we have
$$
\max\left\{
\sup_{(\theta,l)\in\Theta\times\mathbb{N}_0}
\left\{\sup_{x\in\mathsf{X}}
\frac{V_{\theta,l}(x)}{V_{\theta}(x)} 
\right\},
\sup_{(\theta,l)\in\Theta\times\mathbb{N}}
\left\{\sup_{x\in\mathsf{X}}
\frac{V_{\theta,l-1}(x)}{V_{\theta,l}(x)} 
\right\}
\right\}
\leq C.
$$
In addition
$$
\max\left\{
\sup_{(\theta,x)\in\Theta\times\mathbb{N}\times\mathsf{X}}\frac{|H_{0}(\theta,x)|}{V_{\theta,0}(x)^{1/2}}
\sup_{(\theta,l,x)\in\Theta\times\mathbb{N}\times\mathsf{X}}\frac{|H_{l}(\theta,x)|}{V_{\theta,l-1}(x)^{1/2}},
\sup_{(\theta,l,x)\in\Theta\times\mathbb{N}_0\times\mathsf{X}}\frac{\left|\frac{\partial H_l(\theta,x)}{\partial \theta}\right|}{V_{\theta,l}(x)}
\right\} \leq C.
$$
\label{ass:15}
}
\item{For each $r\in(0,1]$,  there exists a $(\beta,C)\in(0,\infty)^2$ such that for any 
$(\theta,l)\in\Theta\times\mathbb{N}$
$$
\seminorm{K_{\theta,l}-K_{\theta}}_{V_{\theta,l}^r}
\leq C\Delta_l^{\beta}
$$
where $V_{\theta,l}$ is as in \ref{ass:12} and $\beta$ does not depend on $r$. \label{ass:16}}
\item{For each $r\in(0,1]$, there exists a $(\beta,C)\in(0,\infty)^2$,  with $\beta$ as in \ref{ass:16},  such that for any 
$(\theta,l)\in\Theta\times\mathbb{N}_0$
$$
\|\pi_{\theta,l}-\pi_{\theta}\|_{V_{\theta}^r} 
\leq C\Delta_l^{\beta}
$$
where $V_{\theta}$ is as in \ref{ass:14}. \label{ass:17}}
\item{For each $r\in(0,1]$,  there exists a $(\beta,C)\in(0,\infty)^2$,  with $\beta$ as in \ref{ass:16},  such that for any 
$(\theta,l)\in\Theta\times\mathbb{N}$
$$
\max\left\{
|\pi_{\theta}(H_{\theta,l}-H_{\theta})|, 
|K_{\theta}(H_{\theta,l}-H_{\theta,l-1})|_{V_{\theta,l-1}^r},
\left|
\pi_{\theta}\left(\frac{\partial H_{\theta,l}}{\partial \theta}-\frac{\partial H_{\theta}}{\partial \theta}\right)
\right|
\right\}
\leq C\Delta_l^{\beta}
$$
where $V_{\theta,l-1}$ is as in \ref{ass:12}. \label{ass:18}}
\item{There exists a $(\beta,C)\in(0,\infty)^2$,  with $\beta$ as in \ref{ass:16},  such that for any 
$(\theta,l)\in\Theta\times\mathbb{N}$
$$
\left|
\int_{\mathsf{X}}\left\{
\frac{\partial \pi_{\theta,l}(x)}{\partial \theta}H_l(\theta,x)-
\frac{\partial \pi_{\theta}(x)}{\partial \theta}H(\theta,x)
\right\}dx
\right|
\leq C\Delta_l^{\beta}.
$$
\label{ass:19}}
\item{There exists a $(\zeta,C)\in(1/2,1]\times(0,\infty)$ such that for any 
$(\theta,\theta')\in\Theta^2$
$$
\max\Bigg\{\left|
\int_{\mathsf{X}}\left\{
\frac{\partial \pi_{\theta}(x)}{\partial \theta}H(\theta,x) - 
\frac{\partial \pi_{\theta'}(x)}{\partial \theta'}H(\theta',x)
\right\}dx
\right|,
$$
$$
\left|\int_{\mathsf{X}}\left\{
\frac{\partial H(\theta,x)}{\partial \theta}\pi_{\theta}(x) -
\frac{\partial H(\theta',x)}{\partial \theta'}H(\theta',x)\pi_{\theta'}(x)\right\}dx\right|
\Bigg\}
\leq C|\theta-
\theta'|^{\zeta}.
$$
\label{ass:110}}
\item{There exists a $(\zeta,C)\in(1/2,1]\times(0,\infty)$,  with $\zeta$ as in \ref{ass:110},   such that for any 
$(\theta,\theta',l,x)\in\Theta^2\times\mathbb{N}_0\times\mathsf{X}$
$$
|H_l(\theta,x)-H_l(\theta',x)| \leq C|\theta-\theta'|^{\zeta}.
$$
\label{ass:111}}
\item{ For each $r\in(0,1]$, there exists a $(\zeta,C)\in(1/2,1]\times(0,\infty)$,  with $\zeta$ as in \ref{ass:110},   such that for any $(\theta,\theta',l)\in\Theta^2\times\mathbb{N}_0$
$$
\max\left\{
\|\pi_{\theta,l}-\pi_{\theta',l}\|_{V_{\theta,l}^r},
\seminorm{K_{\theta,l}-K_{\theta',l}}_{V_{\theta,l}^r}
\right\} \leq C|\theta-\theta'|^{\zeta}
$$
where $V_{\theta,l}$ is as in \ref{ass:12}.
\label{ass:113}}
\item{Let $\mathcal{D}$ be a metric on $\mathsf{X}\times\mathsf{X}$. For each $r\in(0,1]$ there exists a $C\in(0,\infty)$ such that for any $(\theta,l,x,y,f)\in\Theta\times\mathbb{N}_0\times\mathsf{X}^2\times\mathscr{B}_b(\mathsf{X})$
$$
\left|
\int_{\mathsf{X}}f(u)V_{\theta,l}(u)^r K_{\theta,l}(x,du) - 
\int_{\mathsf{X}}f(u)V_{\theta,l}(u)^r K_{\theta,l}(y,du)
\right| \leq C|f|_{\infty}\mathcal{D}(x,y)
$$
where $V_{\theta,l}$ is as in \ref{ass:12}.
\label{ass:112}}
\item{
There exists a $C\in(0,\infty)$ such that for any $(\theta,l,x,y)\in\Theta\times\mathbb{N}_0\times\mathsf{X}^2$
$$
|H_{l}(\theta,x)-H_{l}(\theta,y)| \leq C\mathcal{D}(x,y)
$$
where $\mathcal{D}$ is as in \ref{ass:112}.
\label{ass:114}}
\end{enumerate}
\end{hypA}

\subsection{Discussion of Assumptions}

We now discuss the assumptions that we have made.  (A\ref{ass:1}) \ref{ass:11}-\ref{ass:14}
are fairly standard in the study of Markov chain algorithms,  such as \cite{andr_moulines,non_linear}.
In the case that $K_{\theta,l}$ is a random walk Metropolis-Hastings (M-H) kernel,  the Lyapunov function $V_{\theta,l}(x)$ is,  up-to a constant,  $\pi_{\theta,l}(x)^{-\phi}$ for some $\phi\in(0,1)$.
(A\ref{ass:1}) \ref{ass:15} can be verified for random walk M-H,  for instance when the consective probability density functions become lighter-tailed as the accuracy of the numerical solver grows.  In previous experience (e.g.~\cite{beskos1}) we have found this to be the case,  at least in practice.
(A\ref{ass:1}) \ref{ass:16}-\ref{ass:17} relates to the convergence as $l$ grows of the kernels and invariant measures.  In the context of (A\ref{ass:1}) \ref{ass:17} bounds of these types (in different norms) have been considered in \cite{beskos1} and we expect that (A\ref{ass:1}) \ref{ass:17} could be used to verify (A\ref{ass:1}) \ref{ass:16}.  (A\ref{ass:1}) \ref{ass:18}-\ref{ass:19}  relate to rates of convergence of appropriate expectations
of $H_{\theta,l}$, its $\theta$ derivatives and quantities associated to the derivatives of $\pi_{\theta}$.
We expect them to be verifiable for suitably regular models,  although this is left for future work.
(A\ref{ass:1}) \ref{ass:110}-\ref{ass:113}  are continuity in $\theta$ assumptions,  that in related guises
have been investigated (in weaker forms) in \cite{andr_moulines,fort} and more-or-less are often adopted in the MSA literature.  Finally (A\ref{ass:1}) \ref{ass:112}-\ref{ass:114} are associated to an appropriate continuity of the Markov kernel and $H_{\theta,l}$.  (A\ref{ass:1}) \ref{ass:112} has been investigated (stronger assumption) in \cite{disc_models} and (A\ref{ass:1}) \ref{ass:114} is a Lipschitz-type condition on $H_{\theta,l}$

\section{Main Result}\label{sec:main}

\subsection{Statement}

\begin{theorem}\label{theo:main}
Assume (A\ref{ass:1}) \ref{ass:11}-\ref{ass:114}.  Then there exists a $C\in(0,\infty)$ such that
for any $(l,\theta^{\star}_l,\theta^{\star}_{l-1})\in\mathbb{N}\times\Theta^2$  
$$
\Sigma_{\boldsymbol{\theta}_l^{\star},l} \leq 
$$
$$
C
\left|\left( \frac{dh_l(\theta^{\star}_l)}{d\theta}+\frac{dh_{l-1}(\theta^{\star}_{l-1})}{d\theta}
\right)^{-1}\right|
\max\{\pi_{\theta^{\star}_l,l}(V_{\theta^{\star}_l,l}),\pi_{\theta^{\star}_{l-1},l-1}(V_{\theta^{\star}_{l-1},l-1})\}
\left(
\Delta_l^{\beta} + |\theta_l^{\star}-\theta_{l-1}^{\star}|^{\zeta} + \check{\pi}_{\boldsymbol{\theta}_l^{\star},l}\left(
\mathcal{D}^2
\right)^{1/2}
\right)
$$
where $\beta\in(0,\infty)$ is as (A\ref{ass:1}) \ref{ass:17}, $\zeta\in(1/2,1]$ is as (A\ref{ass:1}) \ref{ass:110} and $\mathcal{D}$ is as (A\ref{ass:1}) \ref{ass:112}.
\end{theorem}

\begin{proof}
We have that
$$
\Sigma_{\boldsymbol{\theta}_l^{\star},l} = T_1 + T_2
$$
where
\begin{eqnarray*}
 T_1 & = & -\left(2\frac{dh_l(\theta^{\star}_l)}{d\theta}\right)^{-1}
\check{\pi}_{\boldsymbol{\theta}_l^{\star},l}\left(
\widehat{H}_{\theta_l^{\star},l}^2\otimes 1 - 
K_{\theta^{\star}_l,l}\left(\widehat{H}_{\theta_l^{\star},l}\right)^2\otimes 1
\right) + \\
& & 
\left(
\frac{dh_l(\theta^{\star}_l)}{d\theta} + 
\frac{dh_{l-1}(\theta^{\star}_{l-1})}{d\theta}\right)^{-1}
\check{\pi}_{\boldsymbol{\theta}_l^{\star},l}\left(
\widehat{H}_{\theta_l^{\star},l}\otimes \widehat{H}_{\theta_{l-1}^{\star},l-1} - 
K_{\theta^{\star}_l,l}\left(\widehat{H}_{\theta_l^{\star},l}\right)
\otimes K_{\theta^{\star}_{l-1},l-1}\left(\widehat{H}_{\theta_{l-1}^{\star},l-1}\right)\right)\\
T_2 & = & -\left(2\frac{dh_{l-1}(\theta^{\star}_{l-1})}{d\theta}\right)^{-1}
\check{\pi}_{\boldsymbol{\theta}_l^{\star},l}\left(
1\otimes \widehat{H}_{\theta_{l-1}^{\star},l-1}^2 - 
1\otimes K_{\theta^{\star}_{l-1},l-1}\left(\widehat{H}_{\theta_{l-1}^{\star},l-1}\right)^2
\right) +\\
& & 
\left(
\frac{dh_l(\theta^{\star}_l)}{d\theta} + 
\frac{dh_{l-1}(\theta^{\star}_{l-1})}{d\theta}\right)^{-1}
\check{\pi}_{\boldsymbol{\theta}_l^{\star},l}\left(
\widehat{H}_{\theta_l^{\star},l}\otimes \widehat{H}_{\theta_{l-1}^{\star},l-1} - 
K_{\theta^{\star}_l,l}\left(\widehat{H}_{\theta_l^{\star},l}\right)
\otimes K_{\theta^{\star}_{l-1},l-1}\left(\widehat{H}_{\theta_{l-1}^{\star},l-1}\right)\right).
\end{eqnarray*}
$T_1$ and $T_2$ can be dealt with similarly so we consider only $T_1$.

We have
$$
T_1=T_3+T_4
$$
where
\begin{eqnarray*}
T_3 & = & \left\{
\left(
\frac{dh_l(\theta^{\star}_l)}{d\theta} + 
\frac{dh_{l-1}(\theta^{\star}_{l-1})}{d\theta}\right)^{-1} - 
\left(2\frac{dh_l(\theta^{\star}_l)}{d\theta}\right)^{-1}
\right\}\check{\pi}_{\boldsymbol{\theta}_l^{\star},l}\left(
\widehat{H}_{\theta_l^{\star},l}^2\otimes 1 - 
K_{\theta^{\star}_l,l}\left(\widehat{H}_{\theta_l^{\star},l}\right)^2\otimes 1
\right)\\
T_4 & = &
\left(
\frac{dh_l(\theta^{\star}_l)}{d\theta} + 
\frac{dh_{l-1}(\theta^{\star}_{l-1})}{d\theta}\right)^{-1}
\Bigg\{-
\check{\pi}_{\boldsymbol{\theta}_l^{\star},l}\left(
\widehat{H}_{\theta_l^{\star},l}^2\otimes 1 - 
K_{\theta^{\star}_l,l}\left(\widehat{H}_{\theta_l^{\star},l}\right)^2\otimes 1
\right) + \\
& & 
\check{\pi}_{\boldsymbol{\theta}_l^{\star},l}\left(
\widehat{H}_{\theta_l^{\star},l}\otimes \widehat{H}_{\theta_{l-1}^{\star},l-1} - 
K_{\theta^{\star}_l,l}\left(\widehat{H}_{\theta_l^{\star},l}\right)
\otimes K_{\theta^{\star}_{l-1},l-1}\left(\widehat{H}_{\theta_{l-1}^{\star},l-1}\right)\right)
\Bigg\}.
\end{eqnarray*}
$T_3$ is then easily handled using \eqref{eq:lem41} and Corollary \ref{cor:1}, that is:
$$
|T_3| \leq \left|\left(\frac{dh_l(\theta^{\star}_l)}{d\theta}+\frac{dh_{l-1}(\theta^{\star}_{l-1})}{d\theta}
\right)^{-1}\right|\pi_{\theta^{\star}_l,l}(V_{\theta^{\star}_l,l})
\left(
\Delta_l^{\beta} + |\theta_l^{\star}-\theta_{l-1}^{\star}|^{\zeta}\right).
$$
For $T_4$ one can use \eqref{eq:lem41} and Lemmata \ref{lem:av}, \ref{lem:av1}  to give
$$
|T_4| \leq  \left|\left(\frac{dh_l(\theta^{\star}_l)}{d\theta}+\frac{dh_{l-1}(\theta^{\star}_{l-1})}{d\theta}
\right)^{-1}\right|\pi_{\theta^{\star}_l,l}(V_{\theta^{\star}_l,l})
\left(
\Delta_l^{\beta} + |\theta_l^{\star}-\theta_{l-1}^{\star}|^{\zeta} + \check{\pi}_{\boldsymbol{\theta}_l^{\star},l}\left(
\mathcal{D}^2
\right)^{1/2}
\right)
$$
and the proof is thus concluded.
\end{proof}

\subsection{Implication of Result}

In the following, we suppose that $\Theta$ is compact and that the conditions of \cite[Theorem 5.5]{andr_moulines1} and \cite[(A1)-(A4)]{fort} hold uniformly in each $l$;  this is simply to minimize the technicalities in the subsqeuent discussion.  In this case we have that
$$
\lim_{n\rightarrow\infty} \mathbb{E}\left[\left(
\gamma_n^{-1/2}\left(\{\theta_{n,l}-\overline{\theta}_{n,l-1}\} - 
\{\theta^{\star}_l-\theta^{\star}_{l-1}\}
\right)
\right)^2\right] = \Sigma_{\boldsymbol{\theta}_l^{\star},l}
$$
so that for sufficiently large $n$ one has that
$$
\mathbb{E}\left[\left(\left(\{\theta_{n,l}-\overline{\theta}_{n,l-1}\} - 
\{\theta^{\star}_l-\theta^{\star}_{l-1}\}
\right)
\right)^2\right] \leq \gamma_n\Sigma_{\boldsymbol{\theta}_l^{\star},l}
$$
at least up-to an arbitrarily small additive constant.  Therefore to choose 
$\gamma_n$ to achieve a suitably small error (i.e.~L.H.S.~of the above equation) and hence mean square error,   one can seek to bound 
$\Sigma_{\boldsymbol{\theta}_l^{\star},l}$ which was achieved in Theorem \ref{theo:main}.  From herein we allow $\gamma_n$ to depend on $l$ and write it as $\gamma_{n,l}$.

In the upper-bound in Theorem \ref{theo:main},  under our assumptions the term:
$$
\left|\left( \frac{dh_l(\theta^{\star}_l)}{d\theta}+\frac{dh_{l-1}(\theta^{\star}_{l-1})}{d\theta}
\right)^{-1}\right|
\max\{\pi_{\theta^{\star}_l,l}(V_{\theta^{\star}_l,l}),\pi_{\theta^{\star}_{l-1},l-1}(V_{\theta^{\star}_{l-1},l-1})\}
$$
will, as $l\rightarrow\infty$,  converge to a constant,  so to understand the use of our bound it is enough to consider
soley the terms
$$
\Delta_l^{\beta} + |\theta_l^{\star}-\theta_{l-1}^{\star}|^{\zeta} + \check{\pi}_{\boldsymbol{\theta}_l^{\star},l}\left(
\mathcal{D}^2
\right)^{1/2}.
$$
The term $|\theta_l^{\star}-\theta_{l-1}^{\star}|^{\zeta}$ relates to the bias ($|\theta_l^{\star}-\theta^{\star}|$) and one often sees in applications (or assumes) that this $|\theta_l^{\star}-\theta_{l-1}^{\star}|^{\zeta}=\mathcal{O}(\Delta_l^{\alpha\zeta})$ for some $\alpha>0$.   Since $\check{\pi}_{\boldsymbol{\theta}_l^{\star},l}\left(
\mathcal{D}^2\right)=\check{\pi}_{\boldsymbol{\theta}_l^{\star},l}\check{K}_{\boldsymbol{\theta}_l^{\star},l}\left(
\mathcal{D}^2\right)$ and the latter has been investigated in the literature (numerically at least) e.g.~\cite{disc_models} this appears to be $\mathcal{O}(\Delta_l^{2\beta})$,  at least for some examples. 
All these results imply that 
$$
\Sigma_{\boldsymbol{\theta}_l^{\star},l} = \mathcal{O}(\Delta_l^{\min\{\alpha\zeta,\beta\}}).
$$
As is often the case in practice (e.g.~\cite{disc_models}),  the cost of computing $\{\theta_{n,l}-\overline{\theta}_{n,l-1}\}$ is 
$$
\mathcal{O}\left(\Delta_l^{-\kappa}\gamma_n^{-1}\right)
$$
for some $\kappa>0$.  The bias is supposed as $\mathcal{O}(\Delta_l^{\alpha})$.  
Let $\epsilon>0$ be a level of precision that is required.
Using standard analysis for muiltilevel Monte Carlo (see e.g.~\cite{beskos1} and the references therein) one can choose
$\gamma_{n,l}=1/n_l$,  with
$n_l = \mathcal{O}(\epsilon^{-2}\Delta_l^{(\min\{\alpha\zeta,\beta\}+\kappa)/2})$,  $L=\mathcal{O}(|\log(\epsilon)|)$,  where we suppose that $\min\{\alpha\zeta,\beta\}>\kappa$.  The mean square error is then $\mathcal{O}(\epsilon^{2})$ with a cost $\mathcal{O}(\epsilon^{-2})$.  If $\min\{\alpha\zeta,\beta\}=\kappa$,  which can
happen then one has to adjust $n_l$ and the associated cost is $\mathcal{O}(\epsilon^{-2}\log(\epsilon)^2)$.
The main point however,  is that our main theorem provides a means to choose the number of steps in the MSE procedure,  which was only possible previously using mathematical conjectures as in \cite{ub_inv}.

\subsubsection*{Acknowledgement}

AJ was supported by CUHK-SZ UDF01003537.

\appendix

\section{Technical Results}\label{app:tech}

We remark that if $\varphi_{\theta,l}\in\mathscr{L}_{V_{\theta,l}^r}$ for some $r\in(0,1]$ then
under (A\ref{ass:1}) \ref{ass:11}-\ref{ass:12}
\begin{equation}\label{eq:pois_eq}
\widehat{\varphi}_{\theta,l}(x) = \sum_{n=0}^{\infty} \left\{
[K_{\theta,l}^n-\pi_{\theta,l}](\varphi_{\theta,l})(x)\right\}
\end{equation}
is a well defined solution of the Poisson equation; see for instance \cite{non_linear} and the references therein.
Let $\mathcal{D}$ be a metric on $\mathsf{X}\times\mathsf{X}$ and $\varphi:\mathsf{X}\rightarrow\mathbb{R}$
then we say that $\varphi\in\textrm{Lip}_{\mathcal{D}}(\mathsf{X})$ if for every $(x,y)\in\mathsf{X}\times\mathsf{X}$
$$
|\varphi(x)-\varphi(y)| \leq |\varphi|_{\textrm{Lip}}\mathcal{D}(x,y).
$$

\begin{lem}\label{lem:pois_cont}
Assume (A\ref{ass:1}) \ref{ass:11}-\ref{ass:12}, \ref{ass:112}.  Then for each $r\in(0,1]$ there exists a $C\in(0,\infty)$
such that for any $(\theta,\theta',l,\varphi_{\theta,l},x,y)\in\Theta^2\times\mathbb{N}_0\times\mathscr{L}_{V_{\theta,l}^r}\cap \text{\emph{Lip}}_{\mathcal{D}}(\mathsf{X})\times\mathsf{X}^2$
$$
\max\left\{|\widehat{\varphi}_{\theta,l}(x)-\widehat{\varphi}_{\theta,l}(y)|,
|K_{\theta,l}(\widehat{\varphi}_{\theta,l})(x)-K_{\theta,l}(\widehat{\varphi}_{\theta,l})(y)|
\right\}
\leq C(1+|\varphi_{\theta,l}|_{\text{\emph{Lip}}})\mathcal{D}(x,y).
$$
where $\mathcal{D}(x,y)$ is as (A\ref{ass:1}) \ref{ass:112}.
\end{lem}

\begin{proof}
We give the proof for $|\widehat{\varphi}_{\theta,l}(x)-\widehat{\varphi}_{\theta,l}(y)|$ only
as the other bound can be proved in almost the same way.
We have that 
$$
\widehat{\varphi}_{\theta,l}(x)-\widehat{\varphi}_{\theta,l}(y) = 
\sum_{n=0}^{\infty} \left\{
[K_{\theta,l}^n-\pi_{\theta,l}](\varphi_{\theta,l})(x) - 
[K_{\theta,l}^n-\pi_{\theta,l}](\varphi_{\theta,l})(y)
\right\}.
$$
The summand,  call it $S_n$,  is now considered.  As the case $n=0$ in the above is straightforward
we consider $n\in\mathbb{N}$ and note that $S_n$
can be written as
$$
S_n = \int_{\mathsf{X}} [K_{\theta,l}^{n-1}-\pi_{\theta,l}](\varphi_{\theta,l})(u)K_{\theta,l}(x,du) -
\int_{\mathsf{X}} [K_{\theta,l}^{n-1}-\pi_{\theta,l}](\varphi_{\theta,l})(u)K_{\theta,l}(y,du).
$$
As
$$
[K_{\theta,l}^{n-1}-\pi_{\theta,l}](\varphi_{\theta,l})(u) = 
$$
$$
V_{\theta,l}(u)^r
|\varphi_{\theta,l}|_{V_{\theta,l}^r}|[K_{\theta,l}^{n-1}-\pi_{\theta,l}](\varphi_{\theta,l}/|\varphi_{\theta,l}|_{V_{\theta,l}^r})|_{V_{\theta,l}^r}
\frac{[K_{\theta,l}^{n-1}-\pi_{\theta,l}](\varphi_{\theta,l}/|\varphi_{\theta,l}|_{V_{\theta,l}^r})(u)}
{V_{\theta,l}(u)^r|\varphi_{\theta,l}|_{V_{\theta,l}^r}|[K_{\theta,l}^{n-1}-\pi_{\theta,l}](\varphi_{\theta,l}/|\varphi_{\theta,l}|_{V_{\theta,l}^r})|_{V_{\theta,l}^r}}
$$
by (A\ref{ass:1}) \ref{ass:112} we have that
$$
|S_n| \leq C|\varphi_{\theta,l}|_{V_{\theta,l}^r}|[K_{\theta,l}^{n-1}-\pi_{\theta,l}](\varphi_{\theta,l}/|\varphi_{\theta,l}|_{V_{\theta,l}^r})|_{V_{\theta,l}^r}
\mathcal{D}(x,y).
$$
By (A\ref{ass:1}) \ref{ass:11}-\ref{ass:12}~it follows that for any $n\in\mathbb{N}$,  $|\varphi_{\theta,l}|\leq V_{\theta,l}^r$
\begin{equation}
\sup_{(\theta,l)\in\Theta\times\mathbb{N}_0}|\{K_{\theta,l}^n-\pi_{\theta,l}\}(\varphi_{\theta,l})|_{V_{\theta,l}^r}
\leq C\rho^n\label{eq:lem12}
\end{equation}
with $\rho\in(0,1)$.    So that 
$$
|S_n| \leq C|\varphi_{\theta,l}|_{V_{\theta,l}^r}\rho^{n-1}\mathcal{D}(x,y)
$$
and from here,  the proof is clear.
\end{proof}

\begin{lem}\label{lem:pois_perturb}
Assume (A\ref{ass:1}) \ref{ass:11}-\ref{ass:18}. 
Then for each $r\in(0,1]$, there exists a $C\in(0,\infty)$ such that
for any $(\theta,l,\varphi_{\theta,l},\varphi_{\theta,l-1})\in\Theta\times\mathbb{N}\times\mathscr{L}_{V_{\theta,l-1}^r}^2$
$$
\max\left\{|\widehat{\varphi}_{\theta,l}-\widehat{\varphi}_{\theta,l-1}|_{V_{\theta,l}^r}, 
|K_{\theta,l}(\widehat{\varphi}_{\theta,l})-K_{\theta,l-1}(\widehat{\varphi}_{\theta,l-1})|_{V_{\theta,l}^r} 
\right\}
\leq 
$$
$$
C
\left(
\{|\varphi_{\theta,l-1}|_{V_{\theta,l}^r}
+ |\varphi_{\theta,l}-\varphi_{\theta,l-1}|_{V_{\theta,l-1}^r}
\}
\Delta_l^{\beta} +
|
K_{\theta}\left(
\varphi_{\theta,l}-\varphi_{\theta,l-1}\right)
|_{V_{\theta,l-1}^r}
\right).
$$
\end{lem}

\begin{proof}
We consider the case $\widehat{\varphi}_{\theta,l}-\widehat{\varphi}_{\theta,l-1}|_{V_{\theta,l}^r}$ only as the other quantity can be proved similarly. 
We give the proof for $r=1$ only; the extension to general $r$ is trivial with only notational complications.
We note that the proof is similar to that of \cite[Propositiom 6.4]{non_linear} except with several notational changes which is why the proof is included.
We have the decomposition by using \cite[Proposition C.2.]{non_linear} and some simple calculations
$$
\widehat{\varphi}_{\theta,l}(x)-\widehat{\varphi}_{\theta,l-1}(x) = T_1 + T_2 + T_3
$$
where
\begin{eqnarray*}
T_1 & = & \sum_{n=1}^{\infty}\sum_{i=0}^{n-1} \left\{K_{\theta,l}^i-\pi_{\theta,l}\right\}
\left(\left\{K_{\theta,l}-K_{\theta,l-1}\right\}
\left(
\left\{K_{\theta,l-1}^{n-i-1}-\pi_{\theta,l-1}\right\}\left(\varphi_{\theta,l}\right)
\right)
\right)\\
T_2 & = & \sum_{n=1}^{\infty}\left\{\pi_{\theta,l-1}-\pi_{\theta,l}\right\}\left(
\left\{K_{\theta,l-1}^n-\pi_{\theta,l-1}\right\}\left(\varphi_{\theta,l}\right)
\right)
\\
T_3 & = & \sum_{n=0}^{\infty}\left\{K_{\theta,l-1}^n-\pi_{\theta,l-1}\right\}\left(\varphi_{\theta,l}-\varphi_{\theta,l-1}\right).
\end{eqnarray*}
We will bound each of these terms in turn.

For $T_1$ we consider the summand, for which one has that
$$
\left\{K_{\theta,l}^i-\pi_{\theta,l}\right\}
\left(\left\{K_{\theta,l}-K_{\theta,l-1}\right\}
\left(
\left\{K_{\theta,l-1}^{n-i-1}-\pi_{\theta,l-1}\right\}\left(\varphi_{\theta,l}\right)
\right)
\right)  = 
$$
$$
|\varphi_{\theta,l}|_{V_{\theta,l-1}}
\left|\left\{K_{\theta,l-1}^{n-i-1}-\pi_{\theta,l-1}\right\}\left(\varphi_{\theta,l}
/|\varphi_{\theta,l}|_{V_{\theta,l-1}}
\right)\right|_{V_{\theta,l-1}}
\times
$$
\begin{equation}
\left\{K_{\theta,l}^i-\pi_{\theta,l}\right\}
\left(\left\{K_{\theta,l}-K_{\theta,l-1}\right\}
\left(V_{\theta,l-1}
\frac{\left\{K_{\theta,l-1}^{n-i-1}-\pi_{\theta,l-1}\right\}\left(\varphi_{\theta,l}/|\varphi_{\theta,l}|_{ V_{\theta,l-1} } \right)}
{
V_{\theta,l-1} 
\left|\left\{K_{\theta,l-1}^{n-i-1}-\pi_{\theta,l-1}\right\}\left(\varphi_{\theta,l}
/|\varphi_{\theta,l}|_{V_{\theta,l}}
\right)\right|_{V_{\theta,l-1}}
}
\right)
\right).\label{eq:lem11}
\end{equation}
Then the R.H.S.~of \eqref{eq:lem11} is equal to
$$
|\varphi_{\theta,l}|_{V_{\theta,l-1}}
\left|\left\{K_{\theta,l-1}^{n-i-1}-\pi_{\theta,l-1}\right\}\left(\varphi_{\theta,l}
/|\varphi_{\theta,l}|_{V_{\theta,l-1}}
\right)\right|_{V_{\theta,l-1}}
\seminorm{K_{\theta,l}-K_{\theta,l-1}}_{V_{\theta,l}}
\times
$$
$$
\left\{K_{\theta,l}^i-\pi_{\theta,l}\right\}
\Bigg(V_{\theta,l}\frac{1}{V_{\theta,l}\seminorm{K_{\theta,l}-K_{\theta,l-1}}_{V_{\theta,l}}}\left\{K_{\theta,l}-K_{\theta,l-1}\right\}
\bigg(V_{\theta,l}\frac{V_{\theta,l-1}}{V_{\theta,l}}
\times
$$
$$
\frac{\left\{K_{\theta,l-1}^{n-i-1}-\pi_{\theta,l-1}\right\}\left(\varphi_{\theta,l}/|\varphi_{\theta,l}|_{ V_{\theta,l-1} } \right)}
{
V_{\theta,l-1} 
\left|\left\{K_{\theta,l-1}^{n-i-1}-\pi_{\theta,l-1}\right\}\left(\varphi_{\theta,l}
/|\varphi_{\theta,l}|_{V_{\theta,l-1}}
\right)\right|_{V_{\theta,l-1}}
}
\bigg)
\Bigg).
$$
Now,  for any $x\in\mathsf{X}$
$$
\left|V_{\theta,l}(x)\frac{V_{\theta,l-1}(x)}{V_{\theta,l}(x)}
\frac{\left\{K_{\theta,l-1}^{n-i-1}-\pi_{\theta,l-1}\right\}\left(\varphi_{\theta,l}/|\varphi_{\theta,l}|_{ V_{\theta,l-1} } \right)(x)}
{
V_{\theta,l-1}(x)
\left|\left\{K_{\theta,l-1}^{n-i-1}-\pi_{\theta,l-1}\right\}\left(\varphi_{\theta,l-1}
/|\varphi_{\theta,l}|_{V_{\theta,l-1}}
\right)\right|_{V_{\theta,l-1}}
}
\right|\leq 
\left\{\sup_{(\theta,l)\in\Theta\times\mathbb{N}}|V_{\theta,l-1}|_{V_{\theta,l}}\right\}
V_{\theta,l}(x)
$$
so that, for any $x\in\mathsf{X}$
$$
\Bigg|\frac{1}{V_{\theta,l}(x)\seminorm{K_{\theta,l}-K_{\theta,l-1}}_{V_{\theta,l}}}\left\{K_{\theta,l}-K_{\theta,l-1}\right\}
\bigg(
V_{\theta,l}\frac{V_{\theta,l-1}}{V_{\theta,l}}\times
$$
$$
\frac{\left\{K_{\theta,l-1}^{n-i-1}-\pi_{\theta,l-1}\right\}\left(\varphi_{\theta,l}/|\varphi_{\theta,l}|_{ V_{\theta,l-1} } \right)}
{
V_{\theta,l-1} 
\left|\left\{K_{\theta,l-1}^{n-i-1}-\pi_{\theta,l-1}\right\}\left(\varphi_{\theta,l}
/|\varphi_{\theta,l}|_{V_{\theta,l-1}}
\right)\right|_{V_{\theta,l-1}}
}
\bigg)(x)
\Bigg|\leq \left\{\sup_{(\theta,l)\in\Theta\times\mathbb{N}}|V_{\theta,l-1}|_{V_{\theta,l}}\right\}.
$$
Therefore,  returning to \eqref{eq:lem11} we have
$$
\frac{|T_1|}{V_{\theta,l}(x)} \leq |\varphi_{\theta,l}|_{V_{\theta,l-1}} 
\left\{\sup_{(\theta,l)\in\Theta\times\mathbb{N}}|V_{\theta,l-1}|_{V_{\theta,l}}\right\}
\sum_{n=1}^{\infty}\sum_{i=0}^{n-1}
|K_{\theta,l-1}^{n-i-1}-\pi_{\theta,l-1}|_{V_{\theta,l}}
\seminorm{K_{\theta,l}-K_{\theta,l-1}}_{V_{\theta,l}}
|K_{\theta,l}^i-\pi_{\theta,l}|_{V_{\theta,l-1}}
$$
Recalling \eqref{eq:lem12},  and by applying (A\ref{ass:1}) \ref{ass:15},  \ref{ass:16} also,  we have
\begin{equation}
\frac{|T_1|}{V_{\theta,l}(x)} \leq C|\varphi_{\theta,l}|_{V_{\theta,l-1}}\Delta_l^{\beta}.
\label{eq:lem13}
\end{equation}

For $T_2$ we have
\begin{eqnarray*}
T_2 & = & \sum_{n\in\mathbb{N}}|\varphi_{\theta,l}|_{V_{\theta,l-1}}\left|\left\{K_{\theta,l-1}^n-\pi_{\theta,l-1}\right\}
\left(\varphi_{\theta,l}/|\varphi_{\theta,l}|_{V_{\theta,l-1}}\right)
\right|_{V_{\theta,l-1}}\times \\ & & 
\left\{\pi_{\theta,l-1} - \pi_{\theta} + \pi_{\theta} - \pi_{\theta,l}\right\}
\left(V_{\theta,l}\frac{V_{\theta,l-1}}{V_{\theta,l}}
\frac{\left\{K_{\theta,l-1}^n-\pi_{\theta,l-1}\right\}\left(\varphi_{\theta,l}/|\varphi_{\theta,l}|_{V_{\theta,l-1}}\right)}
{V_{\theta,l-1}\left|
\left\{K_{\theta,l-1}^n-\pi_{\theta,l-1}\right\}\left(\varphi_{\theta,l}/|\varphi_{\theta,l}|_{V_{\theta,l-1}}\right)
\right|_{V_{\theta,l-1}}}
\right).
\end{eqnarray*}
Then on taking absolute values,  applying the triangular inequality,  \eqref{eq:lem12},  
(A\ref{ass:1}) \ref{ass:15}
and (A\ref{ass:1})
\ref{ass:17} we have
\begin{equation}
\frac{|T_2|}{V_{\theta,l}(x)} \leq C|\varphi_{\theta,l}|_{V_{\theta,l-1}}\Delta_l^{\beta}.
\label{eq:lem14}
\end{equation}

For $T_3$ considering the summand:
$$
\left\{K_{\theta,l-1}^n-\pi_{\theta,l-1}\right\}\left(\varphi_{\theta,l}-\varphi_{\theta,l-1}\right) =
$$
$$
\left\{K_{\theta,l-1}^{n-1}-\pi_{\theta,l-1}\right\}
\left(\left\{K_{\theta,l-1}-K_{\theta}\right\}\left(
\varphi_{\theta,l}-\varphi_{\theta,l-1}
\right) + 
K_{\theta}\left(
\varphi_{\theta,l}-\varphi_{\theta,l-1}
\right)
\right) =
$$
$$
|\varphi_{\theta,l}-\varphi_{\theta,l-1}|_{V_{\theta,l-1}}
\left\{K_{\theta,l-1}^{n-1}-\pi_{\theta,l-1}\right\}
\Bigg(
V_{\theta,l-1}\Bigg\{
\frac{\left\{K_{\theta,l-1}-K_{\theta}\right\}\left(
\{\varphi_{\theta,l}-\varphi_{\theta,l-1}\}/|\varphi_{\theta,l}-\varphi_{\theta,l-1}|_{V_{\theta,l-1}}
\right)}{V_{\theta,l-1}} + 
$$
$$
\frac{K_{\theta}\left(
\{\varphi_{\theta,l}-\varphi_{\theta,l-1}\}/|\varphi_{\theta,l}-\varphi_{\theta,l-1}|_{V_{\theta,l-1}}
\right)}{V_{\theta,l-1}}\Bigg\}
\Bigg).
$$
Thus using (A\ref{ass:1}) \ref{ass:16},  \ref{ass:18},  we have
$$
\frac{|T_3|}{V_{\theta,l}(x)} \leq \frac{V_{\theta,l-1}(x)}{V_{\theta,l}(x)} \sum_{n=0}^{\infty}
|K_{\theta,l-1}^{n-1}-\pi_{\theta,l-1}|_{V_{\theta,l-1}}
\Big\{
|\varphi_{\theta,l}-\varphi_{\theta,l-1}|_{V_{\theta,l-1}}
\seminorm{K_{\theta,l-1}-K_{\theta}}_{V_{\theta,l-1}} + 
$$
$$
|
K_{\theta}\left(
\varphi_{\theta,l}-\varphi_{\theta,l-1}\right)
|_{V_{\theta,l-1}}
\Big\} \leq
$$
\begin{equation}
C\left(
|\varphi_{\theta,l}-\varphi_{\theta,l-1}|_{V_{\theta,l-1}}\Delta_l^{\beta} +
|
K_{\theta}\left(
\varphi_{\theta,l}-\varphi_{\theta,l-1}\right)
|_{V_{\theta,l-1}}
\right).
\label{eq:lem15}
\end{equation}
Combining \eqref{eq:lem13}-\eqref{eq:lem15} we conclude the proof.
\end{proof}

\begin{lem}\label{lem:pois_cont_theta}
Assume (A\ref{ass:1}) \ref{ass:11}-\ref{ass:12}, \ref{ass:111}-\ref{ass:113}.  Then for each $r\in(0,1]$ there exists a $C\in(0,\infty)$
such that for any $(\theta,\theta',l,\varphi_{\theta,l})\in\Theta^2\times\mathbb{N}_0\times\mathscr{L}_{V_{\theta,l}^r}$
$$
\max\left\{
|\widehat{\varphi}_{\theta,l}-\widehat{\varphi}_{\theta',l}|_{V_{\theta,l}^r}, 
|K_{\theta,l}(\widehat{\varphi}_{\theta,l})-K_{\theta',l}(\widehat{\varphi}_{\theta',l})|_{V_{\theta,l}^r}
\right\}
\leq C|\theta-\theta'|^{\zeta}.
$$
\end{lem}

\begin{proof}
The proof is very similar to that of Lemma \ref{lem:pois_perturb} and is hence omitted.
\end{proof}

\begin{lem}\label{lem:2}
Assume (A\ref{ass:1}) \ref{ass:15},  \ref{ass:17}-\ref{ass:110}. Then there exists a $C\in(0,\infty)$ such that
for any $(\theta,\theta',l)\in\Theta^2\times\mathbb{N}_0$  
$$
\left|\frac{\partial h_l(\theta)}{\partial \theta}-\frac{\partial h(\theta')}{\partial \theta'}\right|
\leq C\left(\Delta_l^{\beta}+|\theta-\theta'|^{\zeta}\right)
$$
where $\beta\in(0,\infty)$ is as (A\ref{ass:1}) \ref{ass:17} and $\zeta\in(1/2,1]$ is as (A\ref{ass:1}) \ref{ass:110}.
\end{lem}

\begin{proof}
We have that 
$$
\frac{\partial h_l(\theta)}{\partial \theta}-\frac{\partial h(\theta')}{\partial \theta'} = \sum_{j=1}^4T_j 
$$
where
\begin{eqnarray*}
T_1 & = & \int_{\mathsf{X}}\left\{
\frac{\partial H_{l}(\theta,x)}{\partial \theta}\pi_{\theta,l}(x) - 
\frac{\partial H(\theta,x)}{\partial \theta}\pi_{\theta}(x)
\right\}dx \\
T_2 & = & \int_{\mathsf{X}}\left\{
\frac{\partial \pi_{\theta,l}(x)}{\partial \theta}H_{l}(\theta,x) - 
\frac{\partial \pi_{\theta}(x)}{\partial \theta}H(\theta,x)
\right\}dx\\
T_3 & = & \int_{\mathsf{X}}\left\{
\frac{\partial H(\theta,x)}{\partial \theta}\pi_{\theta}(x) - 
\frac{\partial H(\theta',x)}{\partial \theta'}\pi_{\theta'}(x)
\right\}dx \\
T_4 & = & \int_{\mathsf{X}}\left\{
\frac{\partial \pi_{\theta}(x)}{\partial \theta}H(\theta,x) - 
\frac{\partial \pi_{\theta'}(x)}{\partial \theta'}H(\theta',x)
\right\}dx
\end{eqnarray*}
We will bound $T_1,\dots,T_4$ respectively in order to complete the proof.

For $T_1$ we have the upper-bound
$$
|T_1| \leq 
\left|
\int_{\mathsf{X}}\left\{
\frac{\partial H_{l}(\theta,x)}{\partial \theta}
\{\pi_{\theta,l}(x)-\pi_{\theta}(x)\}\right\}dx\right| + 
\left|
\int_{\mathsf{X}}\left\{
\frac{\partial H_{l}(\theta,x)}{\partial \theta}-\frac{\partial H(\theta,x)}{\partial \theta}
\right\}\pi_{\theta}(x)
dx\right|
$$
For the first term on the R.H.S.~we can use (A\ref{ass:1}) \ref{ass:15}, \ref{ass:17} and for the second term
(A\ref{ass:1}) \ref{ass:18},  to obtain
$$
|T_1| \leq C\Delta_l^{\beta}.
$$
For $T_2$ one can use (A\ref{ass:1}) \ref{ass:19} to yield
$$
|T_2| \leq C\Delta_l^{\beta}.
$$
For $T_3$ and $T_4$ a direct application of (A\ref{ass:1}) \ref{ass:110}
gives
\begin{equation}\label{eq:lem21}
\max\{|T_3|,|T_4|\} \leq C|\theta-\theta'|^{\zeta}.
\end{equation}
Combining the bounds on $T_1,\dots,T_4$ allows one to conclude the proof.
\end{proof}

\begin{cor}\label{cor:1}
Assume (A\ref{ass:1}) \ref{ass:15},  \ref{ass:17}-\ref{ass:110}. Then there exists a $C\in(0,\infty)$ such that
for any $(\theta,\theta',l)\in\Theta^2\times\mathbb{N}$  
$$
\left|\frac{\partial h_l(\theta)}{\partial \theta}-\frac{\partial h_{l-1}(\theta')}{\partial \theta'}\right|
\leq C\left(\Delta_l^{\beta}+|\theta-\theta'|^{\zeta}\right)
$$
where $\beta\in(0,\infty)$ is as (A\ref{ass:1}) \ref{ass:17} and $\zeta\in(1/2,1]$ is as (A\ref{ass:1}) \ref{ass:110}.
\end{cor}

\begin{proof}
Follows from the triangular inequality,  Lemma \ref{lem:2} and the bounds in \eqref{eq:lem21}.
\end{proof}

\begin{lem}\label{lem:av}
Assume (A\ref{ass:1}) \ref{ass:11}-\ref{ass:114}.  Then there exists a $C\in(0,\infty)$ such that
for any $(l,\theta^{\star}_l,\theta^{\star}_{l-1})\in\mathbb{N}\times\Theta^2$  
$$
\max\left\{
\left|\check{\pi}_{\boldsymbol{\theta}_l^{\star},l}\left(
\widehat{H}_{\theta_l^{\star},l}^2\otimes 1 - 
\widehat{H}_{\theta_{l}^{\star},l}\otimes \widehat{H}_{\theta_{l-1}^{\star},l-1}
\right)\right|,
\left|\check{\pi}_{\boldsymbol{\theta}_l^{\star},l}\left(
1\otimes \widehat{H}_{\theta_{l-1}^{\star},l-1}^2- 
\widehat{H}_{\theta_{l}^{\star},l}\otimes \widehat{H}_{\theta_{l-1}^{\star},l-1}
\right)\right|
\right\} \leq 
$$
$$
C\pi_{\theta^{\star}_l,l}(V_{\theta^{\star}_l,l})\left(
\Delta_l^{\beta} + |\theta_l^{\star}-\theta_{l-1}^{\star}|^{\zeta} + \check{\pi}_{\boldsymbol{\theta}_l^{\star},l}\left(
\mathcal{D}^2
\right)^{1/2}
\right)
$$
where $\beta\in(0,\infty)$ is as (A\ref{ass:1}) \ref{ass:17}, $\zeta\in(1/2,1]$ is as (A\ref{ass:1}) \ref{ass:110} and $\mathcal{D}$ is as (A\ref{ass:1}) \ref{ass:112}.
\end{lem}
\begin{proof}
We consider the case of 
$$
\left|\check{\pi}_{\boldsymbol{\theta}_l^{\star},l}\left(
\widehat{H}_{\theta_l^{\star},l}^2\otimes 1 - 
\widehat{H}_{\theta_{l-1}^{\star},l}\otimes \widehat{H}_{\theta_l^{\star},l-1}
\right)\right|
$$
as the other term can be treated by a symmetric argument.  We begin by noting 
that 
$$
\int_{\mathsf{X}\times\mathsf{X}}\left\{\widehat{H}_{\theta_l^{\star},l}(x)^2 - \widehat{H}_{\theta_l^{\star},l}(x)\widehat{H}_{\theta_{l-1}^{\star},l-1}(y)\right\}
\check{\pi}_{\boldsymbol{\theta}_l^{\star},l}\left(d(x,y)\right)
= 
\sum_{j=1}^3 T_j
$$
where
\begin{eqnarray*}
T_1 & = & \int_{\mathsf{X}\times\mathsf{X}}\widehat{H}_{\theta_l^{\star},l}(x)\left\{\widehat{H}_{\theta_l^{\star},l}(x)-
\widehat{H}_{\theta_l^{\star},l-1}(x)\right\}
\check{\pi}_{\boldsymbol{\theta}_l^{\star},l}\left(d(x,y)\right)
\\
T_2 & = & 
\int_{\mathsf{X}\times\mathsf{X}}
\widehat{H}_{\theta_l^{\star},l}(x)\left\{
\widehat{H}_{\theta_l^{\star},l-1}(x) - \widehat{H}_{\theta_l^{\star},l-1}(y)
\right\}
\check{\pi}_{\boldsymbol{\theta}_l^{\star},l}\left(d(x,y)\right)
\\
T_3 & = & \int_{\mathsf{X}\times\mathsf{X}}\widehat{H}_{\theta_l^{\star},l}(x)\left\{
\widehat{H}_{\theta_l^{\star},l-1}(y) - \widehat{H}_{\theta_{l-1}^{\star},l-1}(y)
\right\}
\check{\pi}_{\boldsymbol{\theta}_l^{\star},l}\left(d(x,y)\right).
\end{eqnarray*}
So to prove the result we will consider bounding the three terms $T_1,T_2,T_3$ individually.

For $T_1$ we have that
$$
T_1 = \int_{\mathsf{X}\times\mathsf{X}}
V_{\theta_l^{\star},l}(x)
\frac{\widehat{H}_{\theta_l^{\star},l}(x)}{V_{\theta_l^{\star},l}(x)^{1/2}}\left\{
\frac{\widehat{H}_{\theta_l^{\star},l}(x)-
\widehat{H}_{\theta_l^{\star},l-1}(x)}{V_{\theta_l^{\star},l}(x)^{1/2}}\right\}
\check{\pi}_{\boldsymbol{\theta}_l^{\star},l}\left(d(x,y)\right).
$$
By (A\ref{ass:1}) \ref{ass:15} for any $(\theta,l)\in\Theta\times\mathbb{N}$,  $H_{\theta,l}\in\mathscr{L}_{V_{\theta,l}^{1/2}}$,  so then by using the representation \eqref{eq:pois_eq} and the bound \eqref{eq:lem12} we have that 
\begin{equation}\label{eq:lem41}
\frac{|\widehat{H}_{\theta,l}(x)|}{V_{\theta,l}(x)^{1/2}} \leq C
\end{equation}
where $C$ does not depend on $\theta,l,x$.  Then one can use Lemma \ref{lem:pois_perturb}
to deduce that
$$
|T_1| \leq C \pi_{\theta^{\star}_l,l}(V_{\theta,l})\left(
\{|H_{\theta^{\star}_l,l-1}|_{V_{\theta^{\star}_l,l}^{1/2}}
+ |H_{\theta^{\star}_l,l}-H_{\theta^{\star}_l,l-1}|_{V_{\theta^{\star}_l,l-1}^{1/2}}
\}
\Delta_l^{\beta} +
|
K_{\theta^{\star}_l}\left(
H_{\theta^{\star}_l,l}-H_{\theta^{\star}_l,l-1}\right)
|_{V_{\theta^{\star}_l,l-1}^{1/2}}
\right).
$$
Then using (A\ref{ass:1}) \ref{ass:15} and \ref{ass:18} we have
$$
|T_1| \leq C \Delta_l^{\beta}\pi_{\theta^{\star}_l,l}(V_{\theta^{\star}_l,l}).
$$
For $T_2$ one can combine \eqref{eq:lem41},  Lemma \ref{lem:pois_cont} and Caucy-Schwarz to see that
$$
|T_2| \leq C 
\pi_{\theta^{\star}_l,l}(V_{\theta^{\star}_l,l})^{1/2}
\check{\pi}_{\boldsymbol{\theta}_l^{\star},l}\left(
\mathcal{D}^2
\right)^{1/2}.
$$
For $T_3$ one can use \eqref{eq:lem41} and Lemma \ref{lem:pois_cont_theta}, yielding
$$
|T_3| \leq C 
\pi_{\theta^{\star}_l,l}(V_{\theta^{\star}_l,l}^{1/2})|\theta_l^{\star}-\theta_{l-1}^{\star}|^{\zeta}.
$$
Combining the bounds on $T_1,T_2, T_3$ allows us to conclude.
\end{proof}

\begin{lem}\label{lem:av1}
Assume (A\ref{ass:1}) \ref{ass:11}-\ref{ass:114}.  Then there exists a $C\in(0,\infty)$ such that
for any $(l,\theta^{\star}_l,\theta^{\star}_{l-1})\in\mathbb{N}\times\Theta^2$  
$$
\max\Bigg\{
\left|\check{\pi}_{\boldsymbol{\theta}_l^{\star},l}\left(
K_{_{\theta_l^{\star},l}}(\widehat{H}_{\theta_l^{\star},l})^2\otimes 1 - 
K_{_{\theta_l^{\star},l}}(\widehat{H}_{\theta_{l}^{\star},l})\otimes 
K_{_{\theta_{l-1}^{\star},l-1}}(\widehat{H}_{\theta_{l-1}^{\star},l-1})
\right)\right|,
$$
$$
\left|\check{\pi}_{\boldsymbol{\theta}_l^{\star},l}\left(
1\otimes K_{_{\theta_{l-1}^{\star},l-1}}(\widehat{H}_{\theta_{l-1}^{\star},l-1})^2- 
K_{_{\theta_l^{\star},l}}(\widehat{H}_{\theta_{l}^{\star},l})\otimes 
K_{_{\theta_{l-1}^{\star},l-1}}(\widehat{H}_{\theta_{l-1}^{\star},l-1})
\right)\right|
\Bigg\} \leq 
$$
$$
C\pi_{\theta^{\star}_l,l}(V_{\theta^{\star}_l,l})\left(
\Delta_l^{\beta} + |\theta_l^{\star}-\theta_{l-1}^{\star}|^{\zeta} + \check{\pi}_{\boldsymbol{\theta}_l^{\star},l}\left(
\mathcal{D}^2
\right)^{1/2}
\right)
$$
where $\beta\in(0,\infty)$ is as (A\ref{ass:1}) \ref{ass:17}, $\zeta\in(1/2,1]$ is as (A\ref{ass:1}) \ref{ass:110} and $\mathcal{D}$ is as (A\ref{ass:1}) \ref{ass:112}.
\end{lem}

\begin{proof}
This is the same as Lemma \ref{lem:av}.
\end{proof}

\end{document}